\documentclass[11pt]{article}

\usepackage[tbtags]{amsmath}
\usepackage{amssymb}
\usepackage{amsthm}
\usepackage[misc]{ifsym}
\usepackage{cases}
\usepackage{cite}
\usepackage{mathrsfs}
\usepackage{xcolor}
\usepackage{graphicx}
\usepackage{float}
\usepackage{color}
\usepackage{hyperref}
\usepackage{amsmath}

\hypersetup{hypertex=true,
	colorlinks=true,
	linkcolor=red,
	anchorcolor=blue,
	citecolor=blue}

\numberwithin{equation}{section}
\normalsize

\title{\bf MP and DPP for Mean-Variance Portfolio Selection Problem with Poisson Jumps, Recursive Utility and Their Relationship
	\thanks{This work is financially supported by the National Key R\&D Program of China (2022YFA1006104), National Natural Science Foundations of China (12471419, 12271304), and Shandong Provincial Natural Science Foundation (ZR2024ZD35).}}
\author{\normalsize
	Qiyue Zhang\thanks{\it School of Mathematics, Shandong University, Jinan 250100, P.R. China, E-mail: qiyuezhang@mail.sdu.edu.cn},\quad
    Jingtao Shi\thanks{\it Corresponding author, School of Mathematics, Shandong University, Jinan 250100, P.R. China, E-mail: shijingtao@sdu.edu.cn}}
\date{}

\newtheorem{mythm}{Theorem}[section]
\newtheorem{mydef}{Definition}[section]
\newtheorem{mylem}{Lemma}[section]
\newtheorem{Remark}{Remark}[section]

\begin{document}
	
\maketitle

\noindent{\bf Abstract:}\quad
	In this paper, the mean-variance portfolio selection problem with Poisson jumps are studied, where the recursive utility is given by the solution to a backward stochastic differential equation with Poisson jumps. Both the maximum principle and dynamic programming principle are applied to solve this problem, and their relationship is also investigated. The optimal portfolio and efficient frontier of Markowitz's type are derived using both methods. A comparison of efficient frontiers obtained in this paper and in the framework without jumps is conducted.
	
	\vspace{2mm}
	
	
	\vspace{2mm}
	
	
\section{Introduction}

The mean-variance portfolio selection problem can be traced back to the pioneering work of Nobel Prize-winning economist Harry Markowitz \cite{Ma52}. By constructing a portfolio, the investor aims to maximize the mean and minimize the variance of his/her total asset at the terminal time. The construction of portfolios involves a dual objective: maximizing the expected return and minimizing the risk  at the terminal date. However, traditional stochastic optimal control problem focuses on the single-objective optimization. In 2000, Zhou and Li \cite{ZL00} embedded the mean-variance portfolio selection problem into a stochastic {\it linear-quadratic} (LQ) problem, and then solved it to obtain the optimal portfolio and efficient frontier. Another tool for solving this two-objective problem is to consider it as a stochastic optimal control problem of mean field type: Anderson and Djehiche \cite{AD11} proposed this method in 2011.

Compared to traditional time-additive expected utilities, recursive utilities offer a more advanced theoretical framework, thus better informing investors' long-term asset allocation strategies. Duffie and Epstein \cite{DE92} proposed the concept of recursive utilities under the continuous-time framework firstly. El Karoui, Peng and Quenez \cite{KPQ97} used {\it backward stochastic differential equations} (BSDEs) to describe recursive utilities. The recursive utilities represent not only the instantaneous consumption rate but also the future utilities. Therefore, they can more accurately reflect the market fluctuations and have broad applications in finance. More literature about recursive utilities with applications, can refer to \cite{Peng93}, \cite{SS99}, \cite{Ma00}, \cite{EPQ01}, \cite{Sk03}, \cite{KSS13}, \cite{Wu13}, \cite{AM16}, \cite{Aa16}, \cite{Hu17}, \cite{JS18}, \cite{ZL23}, \cite{ZZ25} and the references therein. To the best of our knowledge, Lv et al. \cite{LWZ15} investigated the mean-variance portfolio selection problem with recursive utilities. Recently, Zhang and Li \cite{ZL23}, Zhang and Zhou \cite{ZZ25} discussed the efficient portfolio under the constraint that the mean of wealth is a constant using recursive utilities.

The {\it maximum principle} (MP) and {\it dynamic programming principle} (DPP) are two dominating approaches, to solve stochastic optimal control problems, and have close relationship (Yong and Zhou \cite{YZ99}). Shi and Yu \cite{SY13} researched relationship between MP and DPP for stochastic optimal control problems with recursive utilities. Nie et al. \cite{NSW17} investigated the relationship between MP and DPP for general stochastic optimal control problems with recursive utilities in the  viscosity solution framework, which was recently generalized to the problem with Poission jumps by Wang and Shi \cite{WS24}. Sun et al. \cite{SGZ18} showed the relationship between MP and DPP for Markov regime-switching stochastic optimal control problems with Poisson jumps and recursive utilities. Li \cite{Li23} studied the relationship between MP and DPP for stochastic optimal control problem with recursive utilities under volatility uncertainty. Dong et al. \cite{DMZ24} researched the relationship between MP and DPP for stochastic optimal control problem with recursive utilities and random coefficients.

Motivated by the fact that most real-world systems are not purely continuous, Poisson jumps fill the gap by capturing the discontinuous, infrequent shocks that drive critical behaviors across finance, science, and engineering. Lots of studies on stochastic optimal control problems with Poisson jumps with applications have been conducted by researchers. Framstad et al. \cite{FOS04} obtained a sufficient MP for the stochastic optimal control problem with Poisson jumps and found its financial application. Guo and Xu \cite{GX04} applied the DPP on solving the mean-variance portfolio selection problem with Poisson jumps. Shen and Siu \cite{SS13} viewed the mean-variance portfolio selection problem with Poisson jumps as a mean-field type stochastic optimal control problem and solved it using MP. Shi and Wu \cite{SW11} investigated the relationship between MP and DPP for the stochastic optimal control problem with Poisson jumps. For systems controlled by {\it forward-backward stochastic differential equations with Poisson jumps} (FBSDEPs), \O kensal and Sulem \cite{OS09}, Shi and Wu \cite{SW10} researched MPs, and applied to finance with recursive utilities. Li and Peng \cite{LP09} obtained the DPP for the stochastic optimal control problem with Poisson jumps and recursive utilities. Shi \cite{Shi14} found the relationship between MP and DPP for this problem. 

Very recently, Zhang and Shi \cite{ZS25} used the embedding theorem to the mean-variance portfolio selection problem with Poisson jumps and obtained the relationship between MP and DPP. In this paper, we generalize the problem to the case of recursive utilities. The controlled system is described by an FBSDEP, where the cost functional is defined by the solution to a {\it BSDE with Poisson jumps} (BSDEP). Different from \cite{ZS25}, long-term risk in the entire future consumption stream is considered through recursive utilities. Therefore, the optimal portfolio and efficient frontier we derive better align with market realities. Moreover, \cite{Shi14} and other papers (For example, \cite{ZL23} and \cite{ZZ25}) investigate mean-variance portfolio selection problems with recursive utilities, with the constraint that the mean of wealth is constant. This assumption is not necessary, as it is not inherent to the problem itself, which can be removed from this paper. 

The contribution of this paper can be summarized as follows.

(1) By embedding the problem into a stochastic optimal control problem with Poisson jumps, we can solve it directly, and the relationship between MP and DPP is given. And the constraint on the mean value in \cite{Shi14} can be removed.

(2) Besides the optimal portfolio, the efficient frontier of our problem is obtained, which has no previous result in the literature.

(3) We compare the efficient frontier with recursive utilities in this paper with that in \cite{ZL00}, and analyze their difference, which help us to gain a deeper understanding of efficient frontiers.

The remaining sections of this paper are organized as follows. In section 2, we state our problem and embed our original problem into a stochastic LQ optimal control problem of FBSDEPs. In section 3, the control problem is solved by the MP. The optimal portfolio and efficient frontier are obtained. Specially, the comparison of efficient frontiers is conducted. In section 4, we find the optimal portfolio of the problem by the DPP. In section 5, the relationship between MP and DPP is obtained.

\section{Problem statement}

Let $T > 0$ and $T$ is finite. $(\Omega ,\mathcal{F},\mathbb{P})$ is a complete probability space, equipped with a one-dimensional standard Brownian motion $\left \{ B(t) \right \} _{0\le t \le T}$ and a Poisson random measure $N(\cdot , \cdot)$ independent of $B(\cdot)$ with the intensity measure $\hat{N} (dt,dz)=\lambda(dz)dt$. The compensated Poisson martingale measure is defined as $\tilde{N} (dt,dz):= N(dt,dz)-\lambda(dz)dt$. 

Now we assume a market consisting of two assets: risk-free bond and risky asset. We denote $S_0(t)$ as the price of the risk-free bond at time $t$, which is given by
\begin{equation*}
dS_0 (t)= \rho_tS_0 (t)dt,
\end{equation*}
where $\rho_\cdot> 0$, is a deterministic continuous function on $[0,T]$.
We denote the risky asset price process as $S_1 (t)$, where $t\in [0,T]$. It can be described as following {\it stochastic differential equation with Poisson jumps} (SDEP):
\begin{equation*}
dS_1 (t)= S_1 (t)\Big[\mu_tdt+\sigma_tdB(t)+\int_{\mathbb{R}\setminus\left\{ 0 \right\}} \eta(t,z)\tilde{N} (dt,dz)\Big],
\end{equation*}
where $\mu_\cdot$ and $\sigma_\cdot$ are deterministic continuous functions on $[0,T]$, $\eta(\cdot,\cdot):[0,T]\times \mathbb{R}\mapsto \mathbb{R}$ is a deterministic continuous function. We adopt a common assumption that $\mu_t> \rho_t$, $t\in[0,T]$.

The wealth process of some investor is denoted as $X(\cdot)$, who holds $\theta_0(t)$ units of the risk-free bond and $\theta_1(t)$ units of the risky asset at time $t$. Then
\begin{equation*}
X (t)= \theta_0(t)S_0 (t)+\theta_1(t)S_1 (t),\quad t\in[0,T],
\end{equation*}
where the portfolio selection satisfies self-financing property, and $x>0$ represents the initial wealth. So $X(\cdot)$ satisfies
\begin{equation*}
X (t)= x +\int_0^t \theta_0(s)dS_0 (s)+\int_0^t \theta_1(s)dS_1 (s).
\end{equation*}
We define
\begin{equation*}
v(t):= \theta_1(t) S_1(t),\quad t\in[0,T],
\end{equation*}
as the amount of the wealth invested in the risky asset. So the wealth process $X(\cdot)$ can be described as
\begin{equation}\label{wealth}
\left\{
\begin{aligned}
dX (t)&= \big[\rho_tX(t)+(\mu_t-\rho_t)v(t)\big]dt +\sigma_tv(t)dB(t) +v(t)\int_{\mathbb{R}\setminus\left\{ 0 \right\} } \eta(t,z)\tilde{N} (dt,dz),\\
 X (0)&=x.  
\end{aligned}
\right.
\end{equation}

To ensure the fitness of the cost functional in the subsequent optimal control problem, applying an approach similar to Zhang and Li \cite{ZL23}, we first introduce an estimate of the $L^p$-solution to the BSDEP under given conditions.

\begin{mylem} (Yao \cite{Yao17})
Let $(Y(\cdot),Z(\cdot),K(\cdot,\cdot))$ be the solution to the BSDEP
\begin{equation}\label{bsdep}
Y(t)=\xi +\int_t^T f(s,Y(s),Z(s),K(s,z))ds-\int_t^T Z(s)dW(s)-\int_t^T\int_{\mathbb{R}\setminus\left\{ 0 \right\} } K(s,z)\tilde{N} (dt,dz),
\end{equation}
where $\mathbb{E}[|\xi|^p]< \infty$, and function $f(s,Y(s),Z(s),K(s,z)):\ [0,T]\times\mathbb{R}\times\mathbb{R}\times([0,T]\times\mathbb{R}\setminus\left\{ 0 \right\})\to \mathbb{R}$ satisfy the assumptions in Theorem 2.1 of Yao \cite{Yao17}. Then for some $p \in (1,2)$, there exists a positive constant $C_p$ such that 
\begin{equation} 
\begin{aligned}
&\mathbb E\bigg[\sup_{0\le t\le T}|Y(t)|^p+\Big(\int_0^T|Z(s)|^2ds\Big)^\frac{p}{2}+\int_t^T\int_{\mathbb{R}\setminus\left\{ 0 \right\} } |K(s,z)|^p\tilde{N} (dt,dz) \bigg]\\
&\le C_p \mathbb E\bigg[1+|\xi|^p+\Big(\int_t^T|f(s,0,0,0)|ds\Big)^p\bigg].
\end{aligned}
\end{equation}
\end{mylem}

Below, we present the optimal control problem for this paper, beginning with the definition of the admissible control set.

\begin{mydef}
Any $v(\cdot)$ satisfying $\Big[\mathbb{E}\Big(\int_0^T v^2(t)dt\Big)^\frac{p}{2}\Big]^{1\wedge \frac{1}{p}}< \infty$ is called an admissible control, whose set is denoted by $\mathcal{V}^p[0,T]$, where $p\in (1,2)$.
\end{mydef}

Different from \cite{ZS25} but motivated by \cite{SW10} and \cite{Shi14}, in this paper, we consider the mean-variance portfolio selection problem with Poisson jumps and recursive utilities. For the investor, his/her target is to find an optimal control $\hat{v}(\cdot)$ not only to maximize the terminal expected wealth $\mathbb{E}X(T)$ and minimize the terminal expected risk $\text{Var}X(T)$, but also to maximize a recursive utility at every moment in $[0,T]$.

For classical mean-variance portfolio selection problem, following the same procedure of \cite{ZS25} by the {\it embedding theorem} in \cite{ZL00}, the investor wishes to achieve
\begin{equation}\label{mean-variance cost functional}
\min\limits_{u(\cdot)\in\,\mathcal{U}^p[0,T]}J_0 (u(\cdot)):=\min\limits_{u(\cdot)\in\,\mathcal{U}[0,T]} \left \{ \mathbb{E}\Big[\frac{1}{2}y^{2}(T) \Big] \right \} ,
\end{equation}
where
\begin{equation}\label{state SDE}
\left\{
\begin{aligned}
dy(t)&=\big[\rho_ty(t) +(\mu_t-\rho_t)u(t)+\sqrt{\mu} \beta \rho_t\big]dt+ \sigma_tu(t)dB(t)+u(t)\int_{\mathbb{R}\setminus\left\{ 0 \right\} }\eta(t,z)\tilde{N}(dt,dz),\\
y(0)&=y\equiv \sqrt{\mu}(x-\beta),
\end{aligned}
\right.
\end{equation}
and we have applied the following variable substitution:
\begin{equation}\label{X and y}
\beta:= \frac{\lambda }{2\mu} ,\ y(t):=\sqrt{\mu} (X(t)-\beta),\ u(t)=\sqrt{\mu}v(t),\ t\in[0,T],
\end{equation}
for some Lagrange multiplier $\mu>0$, and $\lambda\equiv1+2\mu\mathbb{E}X(T)$. When $v(\cdot)$ is admissible, it's clear that $u(\cdot)$ is admissible after the variable substitution. For mean-variance portfolio selection problem with recursive utility, the investor wishes to maximize
\begin{equation}\label{recursive cost functional}
J_1 (u(\cdot)): =Y(t)\mid_{t=0},
\end{equation}
where the recursive utility process $Y(\cdot)$ satisfies
\begin{equation}\label{recursive}
Y(t)=\mathbb{E}\Big[-\frac{1}{2}y^{2}(T)+\int_t^T \big[\rho_sX(s)+(\mu_s-\rho_s)v(s)-\gamma Y(s)\big]ds\mid\mathcal{F}_{t}\Big].
\end{equation}
Here, $\gamma\geq0$ is a given constant representing the consumption rate. For recursive utilities in market environment with jumps, see \cite{Ma00}, \cite{SW10}, \cite{AM16}. Now the investor now encounters the following optimization problem: to minimize
\begin{equation}\label{cost functional}
J (u(\cdot )): = -J_1(u(\cdot)).
\end{equation}

In fact, the new wealth process $y(\cdot)$ and the recursive utility process $Y(\cdot)$ satisfy the following FBSDEP:
\begin{equation}\label{FBSDEP}
\left\{
\begin{aligned}
  dy(t)&=\big[\rho_ty(t) +(\mu_t-\rho_t)u(t)+\sqrt{\mu} \beta \rho_t\big]dt+ \sigma_tu(t)dB(t)+u(t)\int_{\mathbb{R}\setminus\left\{ 0 \right\} }\eta(t,z)\tilde{N}(dt,dz),\\
  -dY(t)&=\bigg\{ \rho_t\Big[\frac{y(t)}{\sqrt\mu}+\beta\Big]+(\mu_t-\rho_t)\frac{u(t)}{\sqrt\mu}-\gamma Y(t)\bigg\} dt-Z(t)dB(t)-\int_{\mathbb{R}\setminus\left\{ 0 \right\} }K(t,z)\tilde{N} (dt,dz),\\
  y(0)&=\sqrt\mu(x-\beta),\quad Y(T)=-\frac{1}{2}y^{2}(T).
\end{aligned}
\right.
\end{equation} 

In \eqref{FBSDEP}, the 5-tuple $(y(\cdot), Y(\cdot), Z(\cdot), K(\cdot,\cdot), u(\cdot))$ is  admissible when $u(\cdot)$ is admissible, $\varphi(\cdot)$ is $\mathcal{F}_t$-adapted,
\begin{equation}
\mathbb E\Big(\int_0^T |\varphi(s)|^2ds\Big)^\frac{p}{2}<\infty,\ \varphi=y,\ Y,\ Z,\nonumber
\end{equation}
and
\begin{equation}
\begin{aligned}
\mathbb E\int_0^T \int_{\mathbb{R}\setminus\left\{ 0 \right\} } |K(t,z)|^p\tilde{N}(dt,dz)ds<\infty.\nonumber
\end{aligned}
\end{equation}

Then we can solve this problem by MP and DPP approaches, respectively.
\begin{Remark}
Unlike \cite{Shi14}, \cite{ZL23} and \cite{ZZ25}, which incorporate condition $\mathbb E[X(T)]=a$ (where $a$ is a given constant) to frame the problem within a stochastic optimal control framework, our approach in this paper employs the embedding theorem in \cite{ZL00} to transform the mean-variance problem—a dual-objective optimization—into a classical stochastic optimal control problem, without requiring the constraint $\mathbb E[X(T)]=a$ in the solution process. For further details regarding this point, we refer the reader to \cite{ZS25}.
\end{Remark}
\section{Solving the problem by MP}

We begin by employing the MP (see \cite{OS09}, \cite{SW10}) to tackle the above problem. From \eqref{FBSDEP} and \eqref{cost functional}, we know the Hamiltonian function is given by:
\begin{align}
H(t,y,Y,Z,u,p,q,r,R)&=\big[\rho_ty +(\mu_t-\rho_t)u+\sqrt{\mu} \beta \rho_t\big]p +\sigma_tur_t+u\int_{\mathbb{R}\setminus\left\{ 0 \right\} }\eta(t,z) R(t,z)\lambda(dz)\nonumber\\
&\quad -q\left \{  \rho_t\Big[\frac{y}{\sqrt\mu}+\beta]+(\mu_t-\rho_t)\frac{u}{\sqrt\mu}-\gamma Y\right \},\nonumber
\end{align}
and the adjoint equation is:
\begin{equation}\label{adjoint}
\begin{cases}
dq(t)=-\gamma q(t)dt, \\
-dp(t)=\rho_s\big[p(t)-\frac{1}{\sqrt\mu}q(t)\big]-r(t)dB(t)-\int_{\mathbb{R}\setminus\left\{ 0 \right\} }R(t,z)\tilde{N} (dt,dz),\\
q(0)=1,\quad p(T)=\hat{y}(T)q(T),
\end{cases}
\end{equation}
where $\hat{y}(\cdot)$ is the solution to \eqref{FBSDEP} corresponding to the optimal control $\hat{u}(\cdot$).
Obviously
\begin{equation}\label{q}
q(t)=e^{-\gamma t},\quad t\in[0,T].
\end{equation}
Suppose $p(\cdot)$ has the form
\begin{equation}\label{relation of p, y}
p(\cdot)=[\phi(\cdot)\hat{y}(\cdot)+\psi(\cdot)]q(\cdot),
\end{equation}
where $\phi(\cdot)$ and $\psi(\cdot)$ are determistic differential functions and satisfy $\phi(T)=1, \psi(T)=0$.
Using generalized It\^{o}'s formula (For example, Lemma 2.1 in \cite{ZS25}) to \eqref{relation of p, y}, we get
\begin{align}\label{dp}
dp(t)&=e^{-\gamma t}\Big\{\big[\dot{\phi}(t)+(\rho_t-\gamma)\phi(t)\big]\hat{y}(t)+\phi(t)(\mu_t-\rho_t)\hat{u}(t)+\phi(t)\sqrt\mu\beta\rho_t\nonumber\\
&\quad +\dot{\psi}(t)-\gamma\psi(t)\Big\}dt+ e^{-\gamma t}\phi(t)\sigma_s\hat{u}(t)dB(t)+\int_{\mathbb{R}\setminus\left\{ 0 \right\} }e^{-\gamma t}\phi(t)\eta(t,z)\hat{u}(t)\tilde{N}(dt,dz).
\end{align}
Compare the BSDEPs in \eqref{adjoint} and \eqref{dp}, we obtain:
\begin{equation}\label{compare}
\begin{cases}
 \big[\dot{\phi}(t)+(\rho_t-\gamma)\phi(t)\big]\hat{y}(t)+\phi(t)(\mu_t-\rho_t)\hat{u}(t)+\phi(t)\sqrt\mu\beta\rho_t+\dot{\psi}(t)-\gamma\psi(t)\\
  =-\rho_t\big[\phi(t)\hat{y}(t)+\psi(t)-\frac{1}{\sqrt\mu}\big],\\
  r(t)=\phi(t)\sigma_t\hat{u}(t)e^{-\gamma t},\quad R(t,z)=\phi(t)\hat{u}(t)\eta(t,z)e^{-\gamma t},\quad \text{a.e.}t\in[0,T],\quad \mathbb{P}\text{-a.s.}
\end{cases}
\end{equation}
Substituting \eqref{q}, \eqref{relation of p, y} and the above $r(\cdot),\ R(\cdot,\cdot)$ back into the Hamiltonian function, noting the partial derivative of $H$ with respect to $u$ at $\hat{u}(\cdot)$ is equal to 0, it achieves:
\begin{equation}\label{optimal control of MP}
\hat{u}(t)=\frac{(\rho_t-\mu_t)\big[\phi(t)\hat y(t)+\psi(t)-\frac{1}{\sqrt\mu}\big]}{\phi(t)\Lambda_t},\quad \text{a.e.}t\in[0,T],\quad \mathbb{P}\text{-a.s.},
\end{equation}
where
\begin{equation}
\Lambda_t: =\sigma^{2}_t+\int_{\mathbb{R}\setminus\left \{ 0 \right \} }\eta^{2}(t,z)\lambda(dz).\nonumber
\end{equation}
Another expression for $\hat{u}(\cdot)$ is acquired acccording to the first equality in \eqref{compare}:
\begin{equation}\label{optimal control of MP-another expression}
\hat{u}(t)=\frac{\big[\dot{\phi}(t)+2\rho_t\phi(t)-\gamma\phi(t)\big]\hat{y}(t)+\phi(t)\sqrt\mu\beta\rho_t+\dot{\psi}(t)+(\rho_t-\gamma)\psi(t)-\frac{1}{\sqrt\mu}\rho_t }{\phi(t)(\rho_t-\mu_t)}.
\end{equation}
Comparing \eqref{optimal control of MP} and \eqref{optimal control of MP-another expression}, we get:
\begin{equation}
\begin{aligned}
(\rho_t -\mu_t)^2\Big[\phi(t)\hat{y}(t)+\psi(t)-\frac{1}{\sqrt\mu}\Big]=&\Big\{\big[\dot{\phi}(t)+2\rho_t\phi(t)-\gamma\phi(t)\big]\hat{y}(t)+\phi(t)\sqrt\mu\beta\rho_t\\
&+\dot{\psi}(t)+(\rho_t-\gamma)\psi(t)-\frac{1}{\sqrt\mu}\rho_t \Big\}\Lambda_t,\quad t\in[0,T].\nonumber
\end{aligned}
\end{equation}
Two {\it ordinary differential equations} (ODEs) are derived by comparing the coefficients of the expression above:
\begin{equation}\label{phi}
\begin{cases}
(\rho_t-\mu_t)^2\phi(t)-\big[2\rho_t\phi(t)+\dot{\phi}(t)-\gamma \phi(t)\big]\Lambda_t=0,\\
\phi(T)=1,
\end{cases}
\end{equation}
\begin{equation}\label{psi}
\begin{cases}
(\rho_t-\mu_t)^2(\psi(t)-\frac{1}{\sqrt\mu})-\big[\phi(t)\sqrt\mu\beta\rho_t+(\rho_t-\gamma)\psi(t)+\dot{\psi}(t)-\frac{1}{\sqrt{\mu}}\rho_t\big]\Lambda_t=0,\\
\psi(T)=0.
\end{cases}
\end{equation}
For simplicity of notation, we define
\begin{equation*}
\theta_t:=\frac{(\mu_t-\rho_t)^2}{\Lambda_t},\quad t\in[0,T],
\end{equation*}
then \eqref{phi} and \eqref{psi} become:
\begin{equation}\label{new phi}
\begin{cases}
\dot{\phi}(t)-(\theta_t-2\rho_t+\gamma)\phi(t) =0,\\
\phi(T)=1,
\end{cases}
\end{equation}
\begin{equation}\label{new psi}
\begin{cases}
\dot{\psi}(t)-(\theta_t-\rho_t+\gamma)\psi(t)+\phi(t)\sqrt\mu\beta\rho_t-\frac{1}{\sqrt{\mu}}\rho_t+\frac{\theta_t}{\sqrt\mu}=0,\\
\psi(T)=0.
\end{cases}
\end{equation}
\eqref{new phi} and \eqref{new psi} have explicit solutions:
\begin{equation}\label{phi solution}
\phi(t)=e^{-\int_t^T(\theta_s-2\rho_s+\gamma)ds},
\end{equation}
\begin{equation}\label{psi solution}
\psi(t)=\sqrt\mu\beta\Big[e^{-\int_t^T(\theta_s-2\rho_s+\gamma)ds}-e^{-\int_t^T(\theta_s-\rho_s+\gamma)ds}\Big]+\frac{1}{\sqrt\mu}\int_t^T(\theta_s-\rho_s)e^{-\int_t^T(\theta_u-2\rho_u+\gamma)du}ds.
\end{equation}
Substituting \eqref{phi solution} and \eqref{psi solution} into \eqref{optimal control of MP}, we acquire the following state/wealth feedback form:
\begin{equation}\label{optimal control of MP-state}
\begin{aligned}
\hat{u}(t)=\frac{(\rho_t-\mu_t)\big[\phi(t)\hat y(t)+\psi(t)-\frac{1}{\sqrt\mu}\big]}{\phi(t)\Lambda_t},\quad \text{a.e.}t\in[0,T],\quad \mathbb{P}\text{-a.s.}
\end{aligned}
\end{equation}

We summarize the above analysis into the following theorem.
\begin{mythm}
By the MP, the state/wealth feedback form's optimal control $\hat{u}(\cdot)$ for our problem is given by \eqref{optimal control of MP-state}.
\end{mythm}

As pointed out by \cite{ZL00}, the efficient frontier aims to assist investors in finding the optimal balance between risk and return, which is an essential concept for portfolio selection. Now we try to find out the efficient frontier for our mean-variance portfolio selection problem with Poisson jumps and recursive utilities. First, by applying a variable substitution to \eqref{optimal control of MP-state}, we have
\begin{equation}\label{optimal control of MP-state x}
\begin{aligned}
\hat{v}(t)=\frac{(\rho_t-\mu_t)\big[\phi(t)\sqrt\mu(\hat X(t)-\beta)+\psi(t)-\frac{1}{\sqrt\mu}\big]}{\sqrt\mu\phi(t)\Lambda_t},\quad \text{a.e.}t\in[0,T],\quad \mathbb{P}\text{-a.s.},
\end{aligned}
\end{equation}
where $\hat{X}(\cdot)$ is the solution to \eqref{wealth} corresponding to $\hat{u}(\cdot)$: 
\begin{equation} \label{dX(t)}
\begin{aligned}
d\hat{X}(t) & = \rho_t\hat{X}(t)dt-\theta_t\Big[\hat{X}(t)-\beta+\frac{1}{\sqrt\mu}a(t)\psi(t)-\frac{1}{\mu}a(t)\Big]dt\\
&\quad+\sigma_t\frac{\rho_t-\mu_t}{\Lambda_t}\Big[\hat{X}(t)-\beta+\frac{1}{\sqrt\mu}a(t)\psi(t)-\frac{1}{\mu}a(t)\Big]dB(t)\\
&\quad+\frac{\rho_t-\mu_t}{\Lambda_t}\Big[\hat{X}(t)-\beta+\frac{1}{\sqrt\mu}a(t)\psi(t)-\frac{1}{\mu}a(t)\Big]\int_{\mathbb{R}\setminus\left \{ 0 \right \} }\eta(t,z)\tilde{N}(dt,dz),
\end{aligned}
\end{equation}
where we have defined
\begin{equation}
a(t):=e^{\int_t^T(\theta_v-2\rho_v+\gamma) dv}\equiv\frac{1}{\phi(t)}.
\end{equation}
Then take expectations to both sides, we get
\begin{equation}\label{expectation of X(t)}
\begin{cases}
d\mathbb{E}\hat{X}(t)=\rho_t\mathbb{E}\hat{X}(t)-\theta(t)\Big[\mathbb{E}\hat{X}(t)-\beta+\frac{1}{\sqrt\mu}a(t)\psi(t)-\frac{1}{\mu}a(t)\Big]dt\\
\mathbb{E}\hat{X}(0)=x.
\end{cases}
\end{equation}
Define
\begin{equation}
C(t):=-\beta+\frac{1}{\sqrt\mu}a(t)\psi(t)-\frac{1}{\mu}a(t),
\end{equation}
then the solution of \eqref{expectation of X(t)} is:
\begin{equation} \label{EX(t) solution}
\mathbb{E}\hat{X}(t)=e^{\int_0^t(\rho_s-\theta(s))ds}\Big[x-\int_0^t\theta(s)C(s)e^{-\int_0^t(\rho_u-\theta(u))du }ds\Big]. 
\end{equation}
Using \eqref{dX(t)}, applying generalized It\^{o}'s formula to $\hat{X}^2(\cdot)$ and taking expectations, we get
\begin{equation}\label{dEX^2(t)}
\begin{cases}
d\mathbb{E}\hat{X}^2(t)=\big[(2\rho_t-\theta(t))\mathbb{E}\hat{X}^2(t)+\theta(t)C^2(t)\big]dt,\\
\mathbb{E}\hat{X}^2(0)=x^2.
\end{cases}
\end{equation}
The solution of \eqref{dEX^2(t)} is
\begin{equation}\label{EX^2(t) solution}
\mathbb{E}\hat{X}^2(t)=e^{\int_0^t 2\rho_s-\theta(s)ds}\Big[x^2+\int_0^t\theta(s)C^2(s)e^{-\int_0^s 2\rho_r-\theta(r)dr}ds\Big].
\end{equation}
Finally, using \eqref{EX(t) solution} and \eqref{EX^2(t) solution}, we get
\begin{equation}\label{efficient frontier}
\begin{aligned}
 \text{Var}\hat X(t) & = \mathbb{E}\hat{X}^2(t)-[\mathbb{E}\hat{X}(t)]^2\\
&=e^{\int_0^t 2\rho_s-\theta(s)ds}\Big[x^2+\int_0^t\theta(s)C^2(s)e^{-\int_0^s 2\rho_r-\theta(r)dr}ds\Big]-[\mathbb{E}\hat{X}(t)]^2.
\end{aligned}
\end{equation}
To derive an explicit expression for the efficient frontier, we need to express $C(t)$ in terms of $\mathbb{E}\hat{X}(t)$. From \eqref{EX(t) solution}, after a simple solution of the integral equation for $C(t)$, we obtain
\begin{equation}\label{C(t)}
C(t)=\frac{\rho_t-\theta(t)}{\theta(t)}\mathbb{E}\hat{X}(t)-\frac{1}{\theta(t)}\frac{\mathrm{d} }{\mathrm{d} t}\mathbb{E}\hat{X}(t). 
\end{equation}
Substituting \eqref{C(t)} into \eqref{efficient frontier}, we find out
\begin{equation}\label{recursive efficient frontier}
\text{Var}\hat X(t)=e^{\int_0^t(2\rho_s-\theta(s))ds }x^2+\int_0^t\frac{e^{\int_{s}^{t}(2\rho_u-\theta(u))du }}{\theta(s)}\Big[(\rho_s-\theta(s))\mathbb{E}\hat{X}(s)-\frac{\mathrm{d} }{\mathrm{d} s}\mathbb{E}\hat{X}(s)\Big]^2ds- [\mathbb{E}\hat{X}(t)]^2.
\end{equation}

We summarize the above analysis into the following result.
\begin{mythm}
The efficient frontier of our mean-variance portfolio selection problem with Poisson jumps and recursive utilities \eqref{cost functional} is given by \eqref{recursive efficient frontier}.
\end{mythm}

Now, we wish to compare and analyze the difference in the efficient frontiers derived from mean-variance portfolio selection problems. Zhou and Li \cite{ZL00} provide the classical form of the efficient frontier in the mean-variance portfolio selection problem:
\begin{equation}\label{Zhou efficient frontier}
\text{Var}\hat X(T)=\frac{1}{e^{\int_0^T\theta_0(t)dt}-1} \left(\mathbb{E}\hat{X}(T)-xe^{\int_0^T \rho_tdt}\right)^2,
\end{equation}
where
\begin{equation}
\theta_0(t):=\frac{(\mu_t-\rho_t)^2}{\sigma_t^2}.
\end{equation}
For the mean-variance portfolio selection problem with Poisson jumps \cite{ZS25}, it is known that the efficient frontier can be expressed as follows:
\begin{equation}\label{jump efficient frontier}
\text{Var}\hat X(T)=\frac{1}{e^{\int_0^T\theta(t)dt}-1} \left(\mathbb{E}\hat{X}(T)-xe^{\int_0^T \rho_tdt}\right)^2,
\end{equation}
where
\begin{equation}
\theta(t):=\frac{(\mu_t-\rho_t)^2}{\sigma^{2}_{t}+\int_{\mathbb{R}\setminus\left \{ 0 \right \} }\eta^{2}(t,z)\lambda(dz)}.\nonumber
\end{equation}
Comparing \eqref{Zhou efficient frontier} and \eqref{jump efficient frontier}, it is obviously that the image of both efficient frontiers are parabola. Numerical simulations of both images and further comparative analysis can be found in \cite{ZS25}.

However, we observe \eqref{recursive efficient frontier} differs significantly from \eqref{Zhou efficient frontier} and \eqref{jump efficient frontier}. The fundamental distinction lies in the inclusion of a derivative term with respect to the expectation. In fact, it represents a dynamic efficient frontier. The derivative term emerges as a result of employing recursive utility to characterize the mean-variance portfolio selection problem, where instantaneous investment and consumption are also incorporated. In practical applications, we hypothesize that high-frequency sampling could be utilized to estimate this derivative term, thereby approximating the expression of the dynamic efficient frontier.

\section{Solving the problem by DPP}

Now we use DPP to our problem \eqref{cost functional} (for example, \cite{LP09}). For this target, we have to introduce the following framework. For given $t\in[0,T)$, define the {\it dynamic} cost functional
\begin{equation}\label{recursive cost functional-DPP}
J (t,y;u(\cdot)): =-Y(s)\mid_{s=t},
\end{equation}
where $Y(s),s\in[t,T]$ satisfies
\begin{equation}\label{FBSDEP-DPP}
\left\{
\begin{aligned}
  dy(s)&=\big[\rho_sy(s) +(\mu_s-\rho_s)u(s)+\sqrt{\mu} \beta \rho_s\big]ds+ \sigma_su(s)dB(s)+u(s)\int_{\mathbb{R}\setminus\left\{ 0 \right\} }\eta(s,z)\tilde{N}(ds,dz),\\
  -dY(s)&=\bigg\{ \rho_s\Big[\frac{y(s)}{\sqrt\mu}+\beta\Big]+(\mu_s-\rho_s)\frac{u(s)}{\sqrt\mu}-\gamma Y(s)\bigg\} ds-Z(s)dB(s)-\int_{\mathbb{R}\setminus\left\{ 0 \right\} }K(s,z)\tilde{N} (ds,dz),\\
  y(t)&=\sqrt\mu(x-\beta),\quad Y(T)=-\frac{1}{2}y^{2}(T).
\end{aligned}
\right.
\end{equation} 
Thus, the value function $V(t,y)$ defined by
\begin{equation}\label{value function}
V(t,y):=J(t,y;\hat u(\cdot))=\inf\limits_{u(\cdot)\in\mathcal{U}^p[t,T]}J(t,y;u(\cdot))
\end{equation}
satisfies the HJB equation:
\begin{equation}\label{HJB}
\begin{cases}
-\frac{\partial V}{\partial t}(t,y)+\sup\limits_{u\in\mathbb{R}} G\Big(t,x,-V(t,x),-\frac{\partial V}{\partial y}(t,y),- \frac{\partial^2V}{\partial y^2} (t,y),u\Big)=0,\quad (t,y)\in[0,T)\times\mathbb{R},\\
V(T,y)=\frac{1}{2}y^2(T),\quad y\in\mathbb{R},
\end{cases}
\end{equation}
where the generalized Hamiltonian function $G$ is:
\begin{align}\label{G}
&G\Big(t,y,-V(t,y),-\frac{\partial V}{\partial y}(t,y),- \frac{\partial^2V}{\partial y^2} (t,y),u\Big)\nonumber\\
&:=-\frac{\partial V}{\partial y}(t,y)\big[y\rho_t+u(\mu_t-\rho_t)+\sqrt{\mu}\beta\rho_t\big]-\frac{1}{2} \frac{\partial^2V}{\partial y^2}(t,y)u^2\sigma^2_t\nonumber\\
&\qquad -\int_{\mathbb{R}\setminus\left\{ 0 \right\} }\left[V(t,y+u\eta(t,z))-V(t,y)-u\eta(t,z)\frac{\partial V}{\partial y} (t,y)\right] \lambda(dz)\nonumber\\
&\qquad +\rho_t\Big[\frac{y(t)}{\sqrt\mu}+\beta\Big]+(\mu_t-\rho_t)\frac{u(t)}{\sqrt\mu}+\gamma V(t,y).
\end{align}

We set
\begin{equation}\label{value function of P, Q, R}
V(t,y)=\frac{1}{2}P(t)y^{2}+Q(t)y+R(t),\quad (t,y)\in[0,T]\times\mathbb{R},
\end{equation}
where $P(\cdot), Q(\cdot), R(\cdot)$ are differential functions with $P(T)=1, Q(T)=R(T)=0$. Substituting \eqref{value function of P, Q, R} into \eqref{G}, and then completing square, we get
\begin{align}\label{completing square}
&G\Big(t,y,-V(t,y),-\frac{\partial V}{\partial y}(t,y),- \frac{\partial^2V}{\partial y^2} (t,y),u\Big)\nonumber\\
&=-\frac{1}{2}P(t)\Lambda_t\left[u(t)-\frac{(\rho_t-\mu_t)\big(P(t)y+Q(t)-\frac{1}{\sqrt\mu}\big)}{P(t)\Lambda_t}\right]^{2}+\bigg[\frac{(\mu_t-\rho_t)^2P(t)}{2\Lambda_t}+\frac{1}{2}\gamma P(t)-\rho_tP(t)\bigg]y^2 \nonumber\\
&\quad +\left[\frac{(\rho_t-\mu_t)^2\big(Q(t)-\frac{1}{\sqrt\mu}\big)}{\Lambda_t}+\gamma Q(t)-\sqrt{\mu}\beta\rho_tP(t)-Q(t)\rho_t+\frac{\rho_t}{\sqrt\mu}\right]y\nonumber\\
&\quad+\frac{(\rho_t-\mu_t)^2\big(Q(t)-\frac{1}{\sqrt\mu}\big)^2}{2P(t)\Lambda_t}+Q(t)\sqrt\mu\beta\rho_t+\rho_t\beta+\gamma R(t).
\end{align}
Substituting \eqref{completing square} into \eqref{HJB}, we get
\begin{align}
&\frac{1}{2}P'(t)y^2+Q'(t)y+R'(t)\nonumber\\
&=-\frac{1}{2}P(t)\Lambda_t\left[u(t)-\frac{(\rho_t-\mu_t)\big(P(t)y+Q(t)-\frac{1}{\sqrt\mu}\big)}{P(t)\Lambda_t}\right]^{2}+\bigg[\frac{(\mu_t-\rho_t)^2P(t)}{2\Lambda_t}+\frac{1}{2}\gamma P(t)-\rho_tP(t)\bigg]y^2 \nonumber\\
&\quad +\left[\frac{(\rho_t-\mu_t)^2\big(Q(t)-\frac{1}{\sqrt\mu}\big)}{\Lambda_t}+\gamma Q(t)-\sqrt{\mu}\beta\rho_tP(t)-Q(t)\rho_t+\frac{\rho_t}{\sqrt\mu}\right]y\nonumber\\
&\quad+\frac{(\rho_t-\mu_t)^2\big(Q(t)-\frac{1}{\sqrt\mu}\big)^2}{2P(t)\Lambda_t}+Q(t)\sqrt\mu\beta\rho_t+\rho_t\beta+\gamma R(t)\nonumber\\
&\ge \bigg[\frac{(\mu_t-\rho_t)^2P(t)}{2\Lambda_t}+\frac{1}{2}\gamma P(t)-\rho_tP(t)\bigg]y^2\nonumber\\
&\quad+\left[\frac{(\rho_t-\mu_t)^2\big(Q(t)-\frac{1}{\sqrt\mu}\big)}{\Lambda_t}+\gamma Q(t)-\sqrt{\mu}\beta\rho_tP(t)-Q(t)\rho_t+\frac{\rho_t}{\sqrt\mu}\right]y\nonumber\\
&\quad+\frac{(\rho_t-\mu_t)^2\big(Q(t)-\frac{1}{\sqrt\mu}\big)^2}{2P(t)\Lambda_t}+Q(t)\sqrt\mu\beta\rho_t+\rho_t\beta+\gamma R(t).
\end{align}
From \eqref{completing square}, we know that if and only if the following three ODEs
\begin{equation}\label{ODE P}
\begin{cases}
P'(t)-(\theta_t+\gamma-2\rho_t)P(t) =0,\\
P(T)=1,
\end{cases}
\end{equation}
\begin{equation}\label{ODE Q}
\begin{cases}
Q'(t)-(\theta_t+\gamma-\rho_t)Q(t)+\sqrt{\mu}\beta\rho_tP(t)-\frac{\rho_t}{\sqrt\mu}+\frac{\theta_t}{\sqrt\mu}=0,\\
Q(T)=0,
\end{cases}
\end{equation}
\begin{equation}\label{ODE R}
\begin{cases}
R'(t)-\left[\gamma R(t)+\rho_t\beta+Q(t)\sqrt\mu\beta\rho_t+\frac{\theta_t(Q(t)-1)^2}{2P(t)}\right]=0,\\
R(T)=0,
\end{cases}
\end{equation}
admit solutions, problem \eqref{cost functional} has a unique optimal control
\begin{equation}\label{optimal control of DPP}
\hat{u}(t)=\frac{(\rho_t-\mu_t)\Big(P(t)\hat{y}(t)+Q(t)-\frac{1}{\sqrt\mu}\Big)}{P(t)\Lambda_t},\quad \text{a.e.}t\in[0,T],
\end{equation}
where $\hat{y}(\cdot)$ is the solution to \eqref{FBSDEP-DPP} corresponding to $\hat{u}(\cdot)$.

Notice that \eqref{ODE P}, \eqref{ODE Q} are both first-order ODEs, we can find out their unique solutions:
\begin{equation}\label{P}
P(t)=e^{-\int_t^T(\theta_u-2\rho_u+\gamma)du},\quad t\in[0,T].
\end{equation}
\begin{equation}\label{Q}
\begin{aligned}
Q(t)&=\sqrt\mu\beta\Big[e^{-\int_t^T(\theta_s-2\rho_s+\gamma)ds}-e^{-\int_t^T(\theta_s-\rho_s+\gamma)ds}\Big]\\
    &\quad +\frac{1}{\sqrt\mu}\int_t^T(\theta_s-\rho_s)e^{-\int_t^s(\theta_u-2\rho_u+\gamma)du}ds,\quad t\in[0,T].
\end{aligned}
\end{equation}
Comparing \eqref{phi solution} and \eqref{P}, we discover
\begin{equation}\label{phi and P}
\phi(t)=P(t),\quad t\in[0,T],
\end{equation}
and comparing \eqref{psi solution} and \eqref{Q}, we find
\begin{equation}\label{psi and Q}
\psi(t)=Q(t),\quad t\in[0,T].
\end{equation}

It stands to reason that the expression \eqref{optimal control of DPP} is the same as \eqref{optimal control of MP-state}.

We summarize the above process into the following result.
\begin{mythm}
By the DPP approach, the state/wealth feedback form's optimal control $\hat{u}(\cdot)$ for our mean-variance portfolio selection problem with Poisson jumps and recursive utilities \eqref{cost functional}, is given by \eqref{optimal control of DPP}.
\end{mythm}

\section{Relationship between MP and DPP}

After the calculations in section 3 and section 4, we can verify the following result directly, showing the relationship between MP and DPP for our mean-variance portfolio selection problem with Poisson jumps and recursive utilities.
\begin{mythm}
{\bf (Relationship between MP and DPP)}\quad For problem \eqref{cost functional}, let $\hat{u}(\cdot)$ be an optimal control and $\hat{y}(\cdot)$ is the corresponding optimal state/wealth trajectory satisfying \eqref{FBSDEP}, then the following results hold:
\begin{equation}\label{relation between MP and DPP}
\begin{cases}
p(t)=\frac{\partial V}{\partial y}(t,\hat{y}(t))q(t),\\
r(t)=\frac{\partial^2 V}{\partial y^2}(t,\hat{y}(t))\sigma_t\hat{u}(t)q(t),\\
R(t,z)=\Big[\frac{\partial V}{\partial y}\big(t,\hat{y}(t)+\hat{u}(t)\eta(t,z)\big)-\frac{\partial V}{\partial y}(t,\hat{y}(t))\Big]q(t),\quad \text{a.e.}t\in[0,T],\quad \mathbb{P}\text{-a.s.},
\end{cases}
\end{equation}
where $(p(\cdot),r(\cdot),R(\cdot,\cdot))$ satisfies the adjoint equation \eqref{adjoint}, and $V(\cdot,\cdot)$ is the value function.
\end{mythm}

\begin{proof} It is direct by \eqref{relation of p, y}, \eqref{compare}, \eqref{value function of P, Q, R}, \eqref{phi and P} and \eqref{psi and Q}. We omit the detail.
\end{proof}

\section{Conclusions}

In this paper, we have studied the maximum principle and dynamic programming principle for mean-variance portfolio selection problem with Poisson jumps and recursive utilities. The optimal portfolio is obtained through two methods respectively. What's more, the relationship between MP and DPP is investigated. In particular, a comparative analysis of various efficient frontier expressions in mean-variance frameworks has been conducted to identify the distinctions among them.

In the future, we hope to simplify the expression of the efficient frontier derived in this paper and examine the practical implications of the parameters in the formula, exploring the implications of this efficient frontier for financial practice. Moreover, we will consider a broader range of financial mathematics problems, employing BSDE and optimal control theory to characterize financial products and investment behaviors.

\end{document}